\documentclass[12pt]{article}
\usepackage{amsmath}
\usepackage{amssymb}
\usepackage{amsthm}
\usepackage{verbatim}
\usepackage{mathrsfs}

\newcommand{\R}[0]{\mathbb R}

\newcommand{\Ds}[0]{\mathcal D}

\newtheorem{Th}{Theorem}[section]
\newtheorem{Lemma}{Lemma}[section]
\newtheorem{Prop}[Lemma]{Proposition}

\begin{document}

\title{On the regularity of the solution map of the porous media equation}
\author{H. Inci}

\maketitle

\begin{abstract}
In this paper we consider the incompressible porous media equation in the Sobolev spaces $H^s(\R^2), s > 2$. We prove that for $T > 0$ the time $T$ solution map $\rho_0 \mapsto \rho(T)$ is nowhere locally uniformly continuous. On the other hand we show that the particle trajectories are analytic curves in $\R^2$.
\end{abstract}

\section{Introduction}\label{section_introduction}

The initial value problem for the incompressible porous media equation in the Sobolev space $H^s(\R^2), s > 2$ is given by
\begin{align}
\label{porous}
\begin{split}
 \rho_t + (u \cdot \nabla) \rho &=0 \\
\operatorname{div} u &=0 \\
u &= -\nabla p - \left(\begin{array}{c} 0 \\ \rho \end{array}\right) \\
\rho(0) & = \rho_0
\end{split}
\end{align}
where $\rho:\R \times \R^2 \to \R$ is the density, $u:\R \times \R^2 \to \R^2$ the velocity of the flow, $p:\R \times \R^2 \to \R$ the pressure. Local well-posedness for $\rho$ lying in $C^0([0,T];H^s(\R^2))$ is known -- see \cite{cordoba}. For $T > 0$ we denote by $U_T \subseteq H^s(\R^2)$ the set of initial values $\rho_0$ for which the solution of \eqref{porous} exists longer than time $T$. We can state our main theorem as
\begin{Th}\label{th_nonuniform}
Let $s > 2$ and $T > 0$. Then the time $T$ solution map
\[
 U_T \to H^s(\R^2),\quad \rho_0 \mapsto \Phi_T(\rho_0)=\rho(T)
\]
is nowhere locally uniformly continuous. Here $\rho(T)$ is the value of $\rho$ at time $T$.
\end{Th}

Our method relies on a geometric formulation of \eqref{porous}. This approach was made popular by the work of Arnold \cite{arnold} for the incompressible Euler equations and subsequently by \cite{ebin_marsden}. In the following we will work this out for \eqref{porous}. Taking the divergence in the third equation (Darcy's law) in \eqref{porous} we have
\[
 -\Delta p=\partial_2 \rho
\]
Reexpressing $\nabla p$ gives
\[
 u=\left(\begin{array}{c} u_1 \\ u_2 \end{array}\right) = \left( \begin{array}{c} -(-\Delta)^{-1} \partial_1 \partial_2 \rho \\ -(-\Delta)^{-1} \partial_2^2 \rho - \rho \end{array}\right) = \left( \begin{array}{c} -(-\Delta)^{-1} \partial_1 \partial_2 \rho \\ (-\Delta)^{-1} \partial_1^2 \rho \end{array}\right) = \left( \begin{array}{c} -\mathcal R_1 \mathcal R_2 \rho \\ \mathcal R_1^2 \rho \end{array}\right)
\]
where
\[
 \mathcal R_k=\partial_k (-\Delta)^{-\frac{1}{2}},\quad k=1,2
\]
are the Riesz operators. Applying $-\mathcal R_1 \mathcal R_2$ resp. $\mathcal R_1^2$ to the first equation in \eqref{porous} gives
\begin{equation}\label{porous2}
 u_t + (u \cdot \nabla) u = \left(\begin{array}{c} \left[u \cdot \nabla,-\mathcal R_1 \mathcal R_2\right] \rho \\ \left[u \cdot \nabla,\mathcal R_1^2\right] \rho
\end{array}\right)
\end{equation}
where we use $[A,B]=AB-BA$ for the commutator of operators. We will express \eqref{porous2} in Lagrangian coordinates, i.e. in terms of the flow map of $u$
\[
 \varphi_t = u \circ \varphi,\quad \varphi(0)=\operatorname{id}
\]
where $\operatorname{id}$ is the identity map in $\R^2$. The functional space for $\varphi$ is for $s > 2$ the diffeomorphism group
\[
 \Ds^s(\R^2):=\{ \varphi:\R^2 \to \R^2 \;|\; \varphi-\operatorname{id} \in H^s(\R^2;\R^2) \mbox{ and } \det(d_x\varphi) > 0 \;\forall x \in \R^2 \}
\]
where $H^s(\R^2;\R^2)$ denotes the vector valued Sobolev space. By the Sobolev imbedding theorem $\Ds^s(\R^2)$ consists of $C^1$ diffeomorphisms. Regarding it as an open subset $\Ds^s(\R^2)-\operatorname{id} \subseteq H^s(\R^2;\R^2)$ it is a connected topological group under composition of maps -- see \cite{composition}. The first equation in \eqref{porous} in terms of $\varphi$ reads as
\[
 \rho(t)=\rho_0 \circ \varphi(t)^{-1}
\]
Taking the $t$ derivative of $\varphi_t = u \circ \varphi$ is
\[
 \varphi_{tt}= (u_t + (u \cdot \nabla) u) \circ \varphi
\]
Thus we can write \eqref{porous2} as
\[
 \varphi_{tt} = \left(\begin{array}{c} \left[(\varphi_t \circ \varphi^{-1}) \cdot \nabla,-\mathcal R_1 \mathcal R_2\right] (\rho_0 \circ \varphi^{-1}) \\ \left[(\varphi_t \circ \varphi^{-1}) \cdot \nabla,\mathcal R_1^2\right] (\rho_0 \circ \varphi^{-1})
\end{array}\right) \circ \varphi
\]
or as a first order equation in $\Ds^s(\R^2) \times H^s(\R^2;\R^2)$
\[
 \frac{d}{dt} \left(\begin{array}{c} \varphi \\ v \end{array}\right) =\left(\begin{array}{c} v \\
\left(\begin{array}{c} \left[(v \circ \varphi^{-1}) \cdot \nabla,-\mathcal R_1 \mathcal R_2\right] (\rho_0 \circ \varphi^{-1}) \\ \left[(v \circ \varphi^{-1}) \cdot \nabla,\mathcal R_1^2\right] (\rho_0 \circ \varphi^{-1})
\end{array}\right) \circ \varphi
\end{array}\right)=\left(\begin{array}{c} v \\ F(\varphi,v,\rho_0) \end{array}\right)
\]
We claim that $F(\varphi,v,\rho_0)$ is analytic in its arguments

\begin{Lemma}\label{lemma_analytic}
The map
\begin{align*}
 F:\Ds^s(\R^2) \times H^s(\R^2;\R^2) \times H^s(\R^2) &\to H^s(\R^2;\R^2)\\
(\varphi,v,\rho_0) &\mapsto F(\varphi,v,\rho_0)
\end{align*}
is analytic.
\end{Lemma}
In the following we will use the notation $R_\varphi:g \mapsto g \circ \varphi$ for the composition from the right. Note that $R_\varphi^{-1}g=g \circ \varphi^{-1}$.
\begin{proof}[Proof of Lemma \ref{lemma_analytic}]
In \cite{sqg} it was shown that for $k=1,2$ 
\[
 \Ds^s(\R^2) \times H^s(\R^2) \to H^s(\R^2),\quad (\varphi,f) \mapsto \left(\mathcal R_k(f\circ \varphi^{-1})\right) \circ \varphi = R_\varphi \mathcal R_k R_\varphi^{-1} f
\]
is analytic and also that
\begin{align*}
 &\Ds^s(\R^2) \times H^s(\R^2;\R^2) \times H^s(\R^2) \to H^s(\R^2)\\
&(\varphi,w,f) \mapsto \left([(w \circ \varphi^{-1})\cdot \nabla,\mathcal R_k](f \circ \varphi^{-1})\right) \circ \varphi=R_\varphi [R_\varphi^{-1}w \cdot \nabla,\mathcal R_k]R_\varphi^{-1} f
\end{align*}
is analytic. Using these results we see by writing for $j,k=1,2$
\begin{align*}
 &R_\varphi [R_\varphi^{-1}v \cdot \nabla, \mathcal R_j \mathcal R_k]R_\varphi^{-1}\rho_0= \\
&R_\varphi [R_\varphi^{-1} v \cdot \nabla,\mathcal R_j] \mathcal R_k R_\varphi^{-1} \rho_0 +
R_\varphi \mathcal R_j [R_\varphi^{-1} v \cdot \nabla,\mathcal R_k] R_\varphi^{-1} \rho_0=\\
&R_\varphi [R_\varphi^{-1} v \cdot \nabla,\mathcal R_j]R_\varphi^{-1} R_\varphi \mathcal R_k R_\varphi^{-1} \rho_0 +
R_\varphi \mathcal R_j R_\varphi^{-1} R_\varphi [R_\varphi^{-1} v \cdot \nabla,\mathcal R_k] R_\varphi^{-1} \rho_0
\end{align*}
that $F(\varphi,v,\rho_0)$ is analytic in its arguments.
\end{proof}
By Picard-Lindel\"of we get for any $\rho_0 \in H^s(\R^2)$ local solutions to
\begin{equation}\label{ode}
 \frac{d}{dt} \left(\begin{array}{c} \varphi \\ v \end{array}\right) =\left( \begin{array}{c} v \\ F(\varphi,v,\rho_0) \end{array}\right),\quad \varphi(0)=\operatorname{id}, v(0)=u_0
\end{equation}
By taking $u_0=(-\mathcal R_1 \mathcal R_2 \rho_0,\mathcal R_1^2 \rho_0)$ we claim that we get solutions to \eqref{porous}.

\begin{Prop}\label{prop_solution}
Let $\rho_0 \in H^s(\R^2)$ and let $(\varphi,v)$ be the solution to \eqref{ode} with $u_0=(-\mathcal R_1 \mathcal R_2 \rho_0,\mathcal R_1^2 \rho_0)$ on some time interval $[0,T]$ with $T > 0$. Then 
\[
 u(t)=v(t) \circ \varphi(t)^{-1} \quad \mbox{and} \quad \rho(t)=\rho_0 \circ \varphi(t)^{-1}
\]
is a solution $(\rho,u) \in C^0([0,T];H^s(\R^2) \times H^s(\R^2;\R^2))$ to \eqref{porous}.
\end{Prop}
\begin{proof}
By the properties of the composition we clearly have
\[
 (\rho,u) \in C^0([0,T];H^s(\R^2) \times H^s(\R^2;\R^2))
\]
Define
\[
 w=\left(\begin{array}{c}w_1 \\ w_2 \end{array}\right)=\left(\begin{array}{c} -R_\varphi \mathcal R_1 \mathcal R_2 R_\varphi^{-1} \rho_0 \\ R_\varphi \mathcal R_1^2 R_\varphi^{-1} \rho_0 \end{array}\right) \in C^\infty([0,T];H^s(\R^2;\R^2))
\]
Calculating the $t$ derivative of $w_1$ gives (note that by the Sobolev imbedding we have $C^1$ expressions)
\begin{eqnarray*}
 \frac{d}{dt} w_1 &=& -\frac{d}{dt} R_\varphi \mathcal R_1 \mathcal R_2 R_\varphi^{-1} \rho_0 = -R_\varphi \mathcal R_1 \mathcal R_2 \frac{d}{dt} R_\varphi^{-1} \rho_0 - (\varphi_t \cdot R_\varphi \nabla) \mathcal R_1 \mathcal R_2 R_\varphi^{-1} \rho_0\\
&=&R_\varphi \mathcal R_1 \mathcal R_2 (R_\varphi^{-1} d \rho_0 \cdot R_\varphi^{-1}[d\varphi]^{-1} \cdot R_\varphi^{-1}\varphi_t) -(\varphi_t \cdot R_\varphi \nabla) \mathcal R_1 \mathcal R_2 R_\varphi^{-1} \rho_0\\
&=& R_\varphi \mathcal R_1 \mathcal R_2 (R_\varphi^{-1}\varphi_t \cdot \nabla) R_\varphi^{-1}\rho_0  -(\varphi_t \cdot R_\varphi \nabla) \mathcal R_1 \mathcal R_2 R_\varphi^{-1} \rho_0\\
&=& R_\varphi [\mathcal R_1 \mathcal R_2,(R_\varphi^{-1}v \cdot \nabla)] R_\varphi^{-1}\rho_0 \\
\end{eqnarray*}
showing that $w_1(t)=v_1(t)$ for $t \in [0,T]$ as $w_1(0)=v_1(0)=-\mathcal R_1 \mathcal R_2\rho_0$. Similarly we have $w_2=v_2$. Thus
\[
 u=w \circ \varphi^{-1} = \left(\begin{array}{c} -\mathcal R_1 \mathcal R_2 \rho\\ \mathcal R_1^2 \rho \end{array}\right)
\]
showing the claim.
\end{proof}
On the other hand consider a solution of \eqref{porous} in
\[
 (\rho,u) \in C^0([0,T];H^s(\R^2) \times H^s(\R^2;\R^2)) \cap C^1([0,T];H^{s-1}(\R^2) \times H^{s-1}(\R^2;\R^2))
\]
We know (see \cite{lagrangian}) that there exists a unique $\varphi \in C^1([0,T];\Ds^s(\R^2))$ with 
\[
 \varphi_t = u \circ \varphi,\quad \varphi(0)=\operatorname{id}
\]
We claim that $\varphi$ and $v=\varphi_t$ is a solution to \eqref{ode}. From \eqref{porous} we get
\[
 u=\left(\begin{array}{c} -\mathcal R_1 \mathcal R_2 (\rho_0 \circ \varphi^{-1}) \\ \mathcal R_1^2(\rho_0 \circ \varphi^{-1}) \end{array}\right)
\] 
Consider $u \circ \varphi$. Taking the $t$ derivative we get pointwise
\[
 (u_t + (u \cdot \nabla) u) \circ \varphi = \left(\begin{array}{c} \left[u \cdot \nabla,-\mathcal R_1 \mathcal R_2\right] (\rho_0 \circ \varphi^{-1}) \\ \left[u \cdot \nabla,\mathcal R_1^2\right] (\rho_0 \circ \varphi^{-1})
\end{array}\right) \circ \varphi
\]
with the righthandside continuous with values in $H^s(\R^2;\R^2)$. Hence $v \in C^1([0,T];H^s(\R^2;\R^2))$ with derivative
\[
 v_t = \left(\begin{array}{c} \left[(v\circ \varphi^{-1}) \cdot \nabla,-\mathcal R_1 \mathcal R_2\right] (\rho_0 \circ \varphi^{-1}) \\ \left[(v\circ \varphi^{-1}) \cdot \nabla,\mathcal R_1^2\right] (\rho_0 \circ \varphi^{-1})
\end{array}\right) \circ \varphi
\]
Hence $(\varphi,v)$ is a solution to \eqref{ode} showing uniqueness of solutions for \eqref{porous} by the uniqueness of solutions for ODEs. Together with Proposition \ref{prop_solution} we get therefore the local well-posedness of \eqref{porous}. 

From the ODE formulation \eqref{ode} we immediately get
\begin{Th}\label{th_analytic}
The particle trajectories of the flow determined by \eqref{porous} are analytic curves in $\R^2$.
\end{Th}

\begin{proof}
As \eqref{ode} is analytic we get by ODE theory that
\[
[0,T] \to D^s(\R^2),\quad t \mapsto \varphi(t)
\]
is analytic. Thus evaluation at $x \in \R^2$, giving the trajectory of the particle which is located at $x$ at time zero,
\[
 [0,T] \to \R^2,\quad t \mapsto \varphi(t,x)
\]
is also analytic. 
\end{proof}

\section{Nonuniform dependence}\label{section_nonuniform}

The goal of this section is to prove Theorem \ref{th_nonuniform}. For the proof we need some preparation. Note that we have the following scaling property for \eqref{porous}: Assume that $(\rho(t,x),u(t,x))$ is a solution to \eqref{porous}. Then a simple calculation shows that for $\lambda > 0$
\[
 (\rho_\lambda(t,x),u_\lambda(t,x))=(\lambda\rho(\lambda t,x),\lambda u(\lambda t,x))
\] 
is also a solution to \eqref{porous}. Thus we have for the domain $U_T \subseteq H^s(\R^2)$
\[
 U_{\lambda \cdot T} = \frac{1}{\lambda} \cdot U_T 
\]
and for the solution map $\Phi_T$
\[
 \Phi_{\lambda \cdot T}(\rho_0)=\frac{1}{\lambda} \cdot \Phi_T(\lambda \cdot \rho_0)
\]
Hence we have $\Phi_T(\rho_0)=\frac{1}{T} \cdot \Phi_1(T \cdot \rho_0)$. Therefore it will be enough to prove Theorem \ref{th_nonuniform} for the case $T=1$. For $T=1$ we introduce 
\[
 U:=U_1 \quad \mbox{and} \quad \Phi:=\Phi_1
\]
Similarly we denote by
\[
 \Psi(\rho_0):=\varphi(1;\rho_0)
\]
where $\varphi(1;\rho_0)$ is the value of the $\varphi$ component at time 1 of the solution to \eqref{ode} for the initial values
\[
 \varphi(0)=\operatorname{id},\quad v(0)=(-\mathcal R_1 \mathcal R_2 \rho_0,\mathcal R_1^2 \rho_0)
\]
Thus by analytic dependence on initial values and parameters we have that
\[
 \Psi:U \subseteq H^s(\R^2) \to \Ds^s(\R^2),\quad \rho_0 \mapsto \Psi(\rho_0)
\]
is analytic. For the proof of the main result we will need the following technical lemma

\begin{Lemma}\label{lemma_dense}
There is a dense subset $S \subseteq U (\subseteq H^s(\R^2))$ with the property that for each $\rho_\bullet \in S$ we have: the support of $\rho_\bullet$ is compact and there is $\bar \rho \in H^s(\R^2)$ and $x^\ast \in \R^2$ with $\operatorname{dist}(x^\ast,\operatorname{supp}\rho_\bullet) > 2$ (i.e. the distance of $x^\ast$ to the support of $\rho_\bullet$ is bigger than 2) and
\[
 \left(d_{\rho_\bullet} \Psi(\bar \rho)\right)(x^\ast) \neq 0
\] 
where $d_{\rho_\bullet} \Psi$ is the differential of $\Psi$ at $\rho_\bullet$.
\end{Lemma}

\begin{proof}
First note that 
\[
 \Psi(t \cdot \bar \rho)=\varphi(1;t \cdot \bar \rho)=\varphi(t;\bar \rho)
\]
where the last equality follows from the scaling property discussed above. Taking the $t$ derivative at $t=0$ gives
\[
 d_0 \Psi(\bar \rho) = \left. \frac{d}{dt} \right|_{t=0} \Psi(t \cdot \bar \rho) = \left. \frac{d}{dt} \right|_{t=0} \varphi(t;\bar \rho) = \varphi_t(0,\bar \rho)=(-\mathcal R_1 \mathcal R_2 \bar \rho,\mathcal R_1^2 \bar \rho)
\]
Fix an arbitrary $\rho_\bullet \in U \subseteq H^s(\R^2)$ with compact support. Take $x^\ast \in \R^2$ with $\operatorname{dist}(x^\ast,\operatorname{supp}(\rho_\bullet)) > 2$. Take a $\tilde \rho \in H^s(\R^2)$ with $-\mathcal R_1 \mathcal R_2 \tilde \rho \neq 0$. The operator $-\mathcal R_1 \mathcal R_2$ is translation invariant as it is a Fourier multiplier operator. Therefore we can choose $\bar \rho(\cdot)=\tilde \rho( \cdot + \delta x)$ with
\[
 (-\mathcal R_1 \mathcal R_2 \bar \rho)(x^\ast) \neq 0
\]
Now consider the analytic curve
\[
 \left(d_{t \rho_\bullet}\Psi(\bar \rho)\right)(x^\ast)
\]
which at $t=0$ is different from zero. Therefore there exist $t_n \uparrow 1$ for $n \geq 1$ with
\[
 \left(d_{t_n \rho_\bullet}\Psi(\bar \rho)\right)(x^\ast) \neq 0
\]
So we can put all these $t_n \rho_\bullet$ into $S$. By this construction we see that $S$ is dense in $U$.
\end{proof}

Theorem \ref{th_nonuniform} will follow from

\begin{Prop}\label{prop_nonuniform}
The time $T=1$ solution map
\[
 \Phi:U \subseteq H^s(\R^2) \to H^s(\R^2),\quad \rho_0 \mapsto \Phi(\rho_0)
\]
is nowhere locally uniformly continuous.
\end{Prop}

\begin{proof}
Let $S \subseteq U$ be as in Lemma \ref{lemma_dense}. Take an arbitrary $\rho_\bullet \in S$. In successive steps we will choose $R_\ast > 0$ and prove that
\[
 \Phi:B_R(\rho_\bullet) \subseteq U \to H^s(\R^2)
\]
is not uniformly continuous for any $0 < R \leq R_\ast$. Here we denote by $B_R(\rho_\bullet) \subset H^s(\R^2)$ the ball of radius $R$ around $\rho_\bullet$. As $S$ is dense this is clearly sufficient to prove the proposition.\\
Fix $x^\ast \in \R^2$ and $\bar \rho \in H^s(\R^2)$ with
\[
 m:=|\left(d_{\rho_\bullet}\Psi(\bar \rho)\right)(x^\ast)| > 0
\]
as guaranteed by Lemma \ref{lemma_dense}. Here $|\cdot|$ is the Euclidean norm in $\R^2$. Define $\varphi_\bullet=\Phi(\rho_\bullet)$ and let
\[
 d:=\operatorname{dist}\left(\varphi_\bullet(\operatorname{supp}\rho_\bullet),\varphi_\bullet(B_1(x^\ast))\right) > 0
\]
where $B_1(x^\ast) \subseteq \R^2$. By the Sobolev imbedding we fix $\tilde C > 0$ with
\begin{equation}\label{sobolev}
 ||f||_{C^1} \leq \tilde C ||f||_s
\end{equation}
for all $f \in H^s(\R^2;\R^2)$. Choose $R_1 > 0$ and $C_1 > 0$ with
\begin{equation}\label{below_above}
 \frac{1}{C_1} ||f||_s \leq ||f \circ \varphi^{-1}||_s \leq C_1 ||f||_s
\end{equation}
for all $f \in H^s(\R^2)$ and for all $\varphi \in \Psi(B_{R_1}(\rho_\bullet))$ which is possible due to the continuity of composition -- see \cite{composition}. Using the Sobolev imbedding \eqref{sobolev} we take $0 < R_2 \leq R_1$ and $L > 0$ with
\begin{equation}\label{lipschitz}
|\varphi(x)-\varphi(y)| < L |x-y| \mbox{ and } |\varphi(x)-\varphi_\bullet(x)| < d/4
\end{equation}
for all $x,y \in \R^2$ and $\varphi \in \Psi(B_{R_2}(\rho_\bullet))$.\\
Consider the Taylor expansion of $\Psi$ around $\rho_\bullet$
\[
 \Psi(\rho_\bullet+h)=\Psi(\rho_\bullet) + d_{\rho_\bullet}\Psi(h) + \int_0^1 (1-s) d^2_{\rho_\bullet + s h} \Psi(h,h) \;ds
\]
for $h \in H^s(\R^2)$. In order to estimate $d^2\Psi$ we choose $0 < R_3 \leq R_2$ such that
\[
 ||d^2_{\rho}\Psi(h_1,h_2)||_s \leq K ||h_1||_s ||h_2||_s
\]
and
\[
 ||d^2_{\rho_1}\Psi(h_1,h_2)-d^2_{\rho_2}\Psi(h_1,h_2)||_s \leq K ||\rho_1-\rho_2||_s ||h_1||_s ||h_2||_s
\]
for all $\rho, \rho_1, \rho_2 \in B_{R_3}(\rho_\bullet)$ and for all $h_1, h_2 \in H^s(\R^2)$ which is possible due to the smoothness of $\Psi$. Finally we choose $0 < R_\ast \leq R_3$ such that
\[
 \tilde C K ||\bar \rho||_s R_\ast^2/4 + \tilde C K ||\bar \rho||_s R_\ast < m/4 
\]
Now fix $0 < R \leq R_\ast$. We will construct two sequences of initial values
\[
 (\rho_0^{(n)})_{n \geq 1},(\tilde \rho_0^{(n)})_{n \geq 1} \subseteq B_R(\rho_\bullet)
\]
with $\lim_{n \to \infty} ||\rho_0^{(n)}-\tilde \rho_0^{(n)}||_s = 0$ whereas
\[
 \limsup_{n \to \infty} ||\Phi(\rho_0^{(n)})-\Phi(\tilde \rho_0^{(n)})||_s > 0
\]
Define the radii $r_n=\dfrac{m}{8nL}$ and choose $w_n \in H^s(\R^2)$ with 
\[
 \operatorname{supp} w_n \subseteq B_{r_n}(x^\ast) \quad \mbox{and} \quad ||w_n||_s = R/2
\]
We choose the initial values as
\[
 \rho_0^{(n)}=\rho_\bullet + w_n \quad \mbox{and} \quad \tilde \rho_0^{(n)}=\rho_\bullet + w_n + \frac{1}{n} \bar \rho
\]
For some $N$ we clearly have
\[
 (\rho_0^{(n)})_{n \geq N},(\tilde \rho_0^{(n)})_{n \geq N} \subseteq B_R(\rho_\bullet)
\]
and $\operatorname{supp}w_n \subseteq B_1(x^\ast)$ for $n \geq N$. By taking $N$ large enough we can also ensure 
\begin{equation}\label{small}
 \tilde C K \frac{1}{n} ||\bar \rho||_s < m/4,\quad \forall n \geq N
\end{equation}
Furthermore
\[
 ||\rho_0^{(n)}-\tilde \rho_0^{(n)}||_s = ||\frac{1}{n} \bar \rho||_s \to 0
\]
as $n \to \infty$. We introduce
\[
 \varphi^{(n)}=\Psi(\rho_0^{(n)}) \quad \mbox{and} \quad \tilde \varphi^{(n)}=\Psi(\tilde \rho_0^{(n)})
\]
With this we have
\[
 \Phi(\rho_0^{(n)})=\rho_0^{(n)} \circ (\varphi^{(n)})^{-1} \quad \mbox{and} \quad \Phi(\tilde \rho_0^{(n)}) = \tilde \rho_0^{(n)} \circ (\tilde \varphi^{(n)})^{-1}
\]
Hence
\[
||\Phi(\rho_0^{(n)})-\Phi(\tilde \rho_0^{(n)})||_s = ||(\rho_\bullet+w_n) \circ (\varphi^{(n)})^{-1}-(\rho_\bullet+w_n+\frac{1}{n}\bar \rho) \circ (\tilde \varphi^{(n)})^{-1}||_s
\]
From \eqref{below_above} we conclude
\begin{align*}
& \limsup_{n \to \infty} ||(\rho_\bullet+w_n) \circ (\varphi^{(n)})^{-1}-(\rho_\bullet+w_n+\frac{1}{n}\bar \rho) \circ (\tilde \varphi^{(n)})^{-1}||_s =\\
& \limsup_{n \to \infty} ||(\rho_\bullet+w_n) \circ (\varphi^{(n)})^{-1}-(\rho_\bullet+w_n) \circ (\tilde \varphi^{(n)})^{-1}||_s 
\end{align*}
Note that by \eqref{lipschitz} we have for $n \geq N$
\[
 \operatorname{supp}(\rho_\bullet \circ (\varphi^{(n)})^{-1}),\operatorname{supp}(\rho_\bullet \circ (\tilde \varphi^{(n)})^{-1}) \subseteq \varphi_\bullet(\operatorname{supp}\rho_\bullet)+B_{d/4}(0)
\]
and
\[
 \operatorname{supp}(w_n \circ (\varphi^{(n)})^{-1}),\operatorname{supp}(w_n \circ (\tilde \varphi^{(n)})^{-1}) \subseteq \varphi_\bullet(B_1(x^\ast))+B_{d/4}(0)
\]
where we use $A+B=\{a+b \;|\; a \in A, b \in B\}$. So the $\rho_\bullet$ and $w_n$ terms are supported in disjoint sets. This allows us (see \cite{sqg}) to estimate with a constant $\bar C > 0$
\begin{align*}
 &\limsup_{n \to \infty} ||(\rho_\bullet+w_n) \circ (\varphi^{(n)})^{-1}-(\rho_\bullet+w_n) \circ (\tilde \varphi^{(n)})^{-1}||_s \geq \\
 &\limsup_{n \to \infty} \bar C ||w_n \circ (\varphi^{(n)})^{-1}-w_n \circ (\tilde \varphi^{(n)})^{-1}||_s 
\end{align*}
The goal is to separate the two $w_n$ expressions by showing that their supports are also disjoint in a suitable way. We have
\[
 \varphi^{(n)}=\Psi(\rho_\bullet + w_n)=\Psi(\rho_\bullet)+d_{\rho_\bullet}\Psi(w_n) + \int_0^1 (1-s) d^2_{\rho_\bullet + s w_n}\Psi(w_n,w_n) \;ds 
\]
resp.
\begin{align*}
 \tilde \varphi^{(n)}&=\Psi(\rho_\bullet + w_n + \frac{1}{n}\bar \rho)=\Psi(\rho_\bullet)+d_{\rho_\bullet}\Psi(w_n + \frac{1}{n}\bar \rho) \\
&+ \int_0^1 (1-s) d^2_{\rho_\bullet + s (w_n+\frac{1}{n}\bar \rho)}\Psi(w_n+\frac{1}{n}\bar\rho,w_n+\frac{1}{n}\bar\rho) \;ds 
\end{align*}
Thus
\[
\tilde \varphi^{(n)} - \varphi^{(n)}=d_{\rho_\bullet}\Psi(\frac{1}{n}\bar \rho) + I_1 + I_2 + I_3
\]
where
\[
 I_1=\int_0^1 (1-s) \left(d^2_{\rho_\bullet + s (w_n+\frac{1}{n}\bar \rho)}\Psi(w_n,w_n)-d^2_{\rho_\bullet + s w_n}\Psi(w_n,w_n)\right)\;ds
\]
and
\[
 I_2=2 \int_0^1 (1-s) d^2_{\rho_\bullet + s (w_n+\frac{1}{n}\bar \rho)}\Psi(w_n,\frac{1}{n}\bar\rho)\;ds
\]
and
\[
 I_3= \int_0^1 (1-s) d^2_{\rho_\bullet + s (w_n+\frac{1}{n}\bar \rho)}\Psi(\frac{1}{n}\bar \rho,\frac{1}{n}\bar\rho)\;ds
\]
Using the estimates for $d^2\Psi$ we have
\[
 ||I_1||_s \leq K \frac{1}{n} ||\bar\rho||_s R^2/4,\quad ||I_2||_s \leq 2K \frac{1}{n}||\bar \rho||_s R/2,\quad ||I_3||_s \leq K \frac{1}{n^2} ||\bar \rho||_s^2
\]
Using \eqref{sobolev}, \eqref{small} and by the choice of $R_\ast$ we have
\[
 |I_1(x^\ast)| + |I_2(x^\ast)| + |I_3(x^\ast)| < \frac{m}{2n}
\]
Hence
\[
 |\tilde \varphi^{(n)}(x^\ast)-\varphi^{(n)}(x^\ast)| \geq |\left(d_{\rho_\bullet}\Psi(\bar \rho)\right)(x^\ast)|/n - \frac{m}{2n}=\frac{m}{2n}
\]
We have by \eqref{lipschitz} and the choice of $r_n$
\[
 \operatorname{supp}(w_n \circ (\tilde \varphi^{(n)})^{-1}) \subseteq B_{Lr_n}(\tilde \varphi^{(n)}(x^\ast))=B_{m/(8n)}(\tilde \varphi^{(n)}(x^\ast))
\]
resp.
\[
  \operatorname{supp}(w_n \circ (\varphi^{(n)})^{-1}) \subseteq B_{Lr_n}(\tilde \varphi^{(n)}(x^\ast))=B_{m/(8n)}(\varphi^{(n)}(x^\ast))
\]
which shows that the supports of the two $w_n$ terms are in such a way disjoint that we can separate the $H^s$ norms with a constant (see \cite{sqg}) to get
\begin{align*}
 ||w_n \circ (\varphi^{(n)})^{-1}-w_n \circ (\tilde \varphi^{(n)})^{-1}||_s &\geq \tilde K (||w_n \circ (\varphi^{(n)})^{-1}||_s + ||w_n \circ (\tilde \varphi^{(n)})^{-1}||_s) \\
 &\geq 2\tilde K \frac{1}{C_1} ||w_n||_s=\frac{\tilde K R}{C_1}
\end{align*}
where we used \eqref{below_above}. Therefore
\[
 \limsup_{n \to \infty} ||\Phi(\rho_0^{(n)})-\Phi(\tilde \rho_0^{(n)})||_s \geq \frac{\tilde K \bar C R}{C_1}
\]
whereas
\[
 ||\rho_0^{(n)}-\tilde \rho_0^{(n)}||_s \to 0 \quad \mbox{as} \quad n \to \infty
\]
showing that $\Phi$ is not uniformly continuous on $B_R(\rho_\bullet)$. As $0 < R \leq R_\ast$ is arbitrary the result follows.
\end{proof}

\bibliographystyle{plain}

\flushleft
\author{ Hasan \.{I}nci\\
Ko\c{c} \"Universitesi Fen Fak\"ultesi\\
Rumelifeneri Yolu\\
34450 Sar{\i}yer \.{I}stanbul T\"urkiye\\
        {\it email: } {hinci@ku.edu.tr}
}

\end{document}